\documentclass[11pt,twoside]{article}
\usepackage{amsfonts}
\usepackage{amsmath}
\usepackage{amssymb}
\usepackage{multicol}
\usepackage{graphicx}
\usepackage{stmaryrd}
\usepackage{cite}
\usepackage{epsfig}

\newtheorem{theorem}{Theorem}[section]
\newtheorem{lemma}[theorem]{Lemma}

\newtheorem{example}[theorem]{Example}
\newtheorem{definition}[theorem]{Definition}
\newtheorem{remark}[theorem]{Remark}

\numberwithin{equation}{section}
\newenvironment{proof}[1][Proof]{\noindent\textbf{#1.} }{\hfill $\Box$}
\allowdisplaybreaks

 \makeatletter
\begin{document}
\author{Shuxia Pan\thanks{E-mail: shxpan@yeah.net.} \thanks{Supported by NSF of Gansu
Province of China (1208RJYA004) and the Development Program for Outstanding Young Teachers in Lanzhou University of Technology (1010ZCX019).} \\
{Department of Applied Mathematics, Lanzhou University of Technology,}\\
{Lanzhou, Gansu 730050, People's Republic of China}}
\title{\textbf{Propagation of Delayed Lattice Differential Equations without Local Quasimonotonicity}}\maketitle
\begin{abstract}
This paper is concerned with the traveling wave solutions and asymptotic spreading of delayed lattice differential equations without quasimonotonicity.
The spreading speed is obtained by constructing auxiliary equations and using the theory of lattice differential equations without time delay. The minimal wave speed of invasion traveling wave solutions is established by presenting the existence and nonexistence of traveling wave solutions.


\end{abstract}

\section{Introduction}
\noindent

Lattice dynamical systems are very important to describe some evolutionary processes in life sciences \cite{bell,keener} and phase transitions
\cite{bates2}. For the scalar lattice differential equations, a typical
example is
\begin{equation}
\frac{du_n(t)}{dt}= [\mathcal{D}u]_n(x)
+f(u_n(t)),n\in \mathbb{Z}, t>0, \label{0}
\end{equation}
where $ f:\mathbb{R}\to \mathbb{R},$ and
\[
 [\mathcal{D}u]_n(x)=D\left( u_{n+1}(t)-2u_n(t)+u_{n-1}(t)\right)
\]
with $D>0.$
In the past decades, much attention has been paid to its propagation modes  indexed by traveling wave solutions and asymptotic spreading, see \cite{and,bates2,bell,cahn2,chenguo1,chenguo2,chow2,hsu,keener,mp1,mp2,mp3,shen,tonnelier,zinner1,zinner2}. In particular, to reflect the maturation time of the species under
consideration and the time needed for the signals  to travel along
axons and to cross synapses, time delay was introduced in lattice differential equations, and a delayed version of \eqref{0} is
\begin{equation}
\frac{du_n(t)}{dt}= [\mathcal{D}u]_n(x)
+f(u_n(t),u_n(t-\tau)),n\in \mathbb{Z}, t>0, \label{1}
\end{equation}
in which $\tau\ge 0$ is the time delay. Since Wu and Zou \cite{wuzou}, some results about the existence of traveling wave solutions and the estimation of asymptotic spreading of \eqref{1} have been established, we refer to  \cite{fang1,fang2,hlr,hlz,liang,llp,lvw,mazou2,mazou,thieme,wlr1,weng,zou}.

To better introduce the known results, we assume that
\begin{equation}
f(0,0)=f(K,K)=0, f(u,u)>0, u\in (0, K),
\label{2}
\end{equation}
with some $K>0,$ and $f(u,v)$ is continuous for $u,v\in [0, K].$ To apply comparison principle, the (local) monotonicity of $f(u,v)$ for $v>0$ is needed, for example, see the scalar models in \cite{fang1,fang2,liang,llp,lvw,mazou2,mazou,thieme,wlr1,weng,zou}. Moreover, if the time delay $\tau>0$ is small enough, some results on the existence of traveling wave solutions have been established by exponential order, see Huang et al. \cite{hlz}.

The purpose of this paper is to consider the propagation of \eqref{1} if $f(u,v)$ is not monotone increasing for $v$ near $0.$ For the sake of convenience, we consider the following special form of \eqref{1}
\begin{equation}
\frac{du_n(t)}{dt}= [\mathcal{D}u]_n(x)
+u_n(t) g(u_n(t),u_n(t-\tau)),n\in \mathbb{Z}, t>0, \label{3}
\end{equation}
in which $g:\mathbb{R}^2 \to \mathbb{R}$ satisfies the following assumptions:
\begin{description}
\item[(H1)] $g(1,0)=0, g(u,0)>0, u\in (0,1);$
\item[(H2)] $g(u,v)$ is Lipschitz continuous and strictly monotone decreasing for $u,v\in [0,1],$ and $g(u,0)\to -\infty, u\to \infty;$
\item[(H3)] $g(0,1)>0,$ and there exists $E\in (0,1)$ such that $g(E,E)=0;$
\item[(H4)] if $1>\overline{u}\ge \underline{u}>0$ such that
\[
g(\underline{u}, \overline{u})\le 0, g(\overline{u},\underline{u})\ge 0,
\]
then $\underline{u}=\overline{u}=E.$
\end{description}
By these assumptions, we see that \eqref{3} does not satisfy the (local) quasimonotonicity in  \cite{fang1,fang2,liang,llp,lvw,mazou2,mazou,thieme,wlr1,weng,zou}. And a typical example of $g$ is
\[
g(u,v)=1-u-av, a\in (0,1),
\]
which is a special form of Logistic nonlinearity with time delay.

In what follows, by using the spreading speed of undelayed scalar lattice differential equations and constructing auxiliary equations without time delay, we shall investigate the propagation of \eqref{3}. We first prove that the spreading speed of \eqref{3} is the same as that $g(u,v)=g(u,u)$ by the idea in Lin \cite{lin}, which implies the persistence of spreading speeds of delayed lattice differential equations even if the time delay $\tau$ is large and the equation cannot generate monotone semiflows. Furthermore, we establish the minimal wave speed of \eqref{3} by presenting the existence and nonexistence of traveling wave solutions for all positive wave speed, which is motivated by the results in Lin and Ruan \cite{linruan}.  These traveling wave solutions formulate the successful invasion of one new invader in population dynamics.

In this paper, we shall use the standard  ordering and interval in $\mathbb{R}.$ Let $C(\mathbb{R},\mathbb{R})$ be
\[
C(\mathbb{R},\mathbb{R})=\{u | u:\mathbb{R}\to \mathbb{R} \text{ is uniformly continuous and bounded}\}.
\]
Then $C$ is a Banach space equipped with the standard supremum norm. When $a<b$ is true, denote
\[
C_{[a,b]}=\{u| u\in C, a\le u \le b\}.
\]
If $u\in C^1\left( \mathbb{R},%
\mathbb{R}\right),$ then
\[
u\in C, u'\in C.
\]
For $\mu >0$, define
\[
B_\mu \left( \mathbb{R},\mathbb{R}\right) =\left\{ u(t):u(t)\in C\left( %
\mathbb{R},\mathbb{R}\right) \text{ and }\sup_{t\in
\mathbb{R}}\left| u(t)\right| e^{-\mu \left| t\right| }<\infty
\right\},
\]
then $B_\mu \left( \mathbb{R},\mathbb{R}\right) $
is a Banach space when it is equipped with the norm $\left| \cdot
\right| _\mu $ defined by
\[
\left| u\right| _\mu =\sup_{t\in \mathbb{R}}\left| u(t)\right|
e^{-\mu \left| t\right| }\text{ for }u\in B_\mu
\left(\mathbb{R},\mathbb{R}\right) .
\]
Let $l^{\infty}$ be
\[
l^{\infty}=\{u(n): n\in \mathbb{Z} \text{ and } u(n) \text{ is bounded for all }n\in \mathbb{Z}\}.
\]

We now give the following definition of  traveling wave solutions.
\begin{definition}\label{de1}
{\rm
A \emph{traveling wave solution} of \eqref{3} is a special
solution with form $u_n(t)=\phi (n+ct)$, in which  $c>0$
is the wave speed and $\phi \in C^1\left( \mathbb{R},%
\mathbb{R}\right) $ is the wave profile that propagates in $\mathbb{Z}$.}
\end{definition}
From Definition \ref{de1}, $\phi$ and $c$ must satisfy
\begin{equation}\label{2.1}
c\frac{d\phi (\xi)}{d\xi }=D(\phi (\xi+1)+\phi (\xi-1)-2\phi
(\xi))+\phi(\xi)g(\phi(\xi) ,\phi(\xi- c\tau)),\xi \in \mathbb{R}.
\end{equation}
To better reflect the evolutionary processes, we also require the following asymptotic boundary value condition
\begin{equation}\label{2.2}
\lim_{\xi\to -\infty}\phi (\xi)=0, \lim_{\xi\to \infty}\phi (\xi)=E.
\end{equation}

To index the asymptotic spreading, we also give the following definition.
\begin{definition}\label{de2}
{\rm
Assume that $u_n(t)$ is a nonnegative function for all $n\in\mathbb{N}, t>0$. $c_1>0$ is the spreading speed of $u_n(t)$ if
\begin{description}
\item[(1)] for any $c>c_1,$ $\lim_{t\to\infty}\sup_{|n|>ct}u_n(t)=0;$
\item[(2)] for any $c<c_1,$ $\liminf_{t\to\infty}\inf_{|n|<ct}u_n(t)>0.$
\end{description}}
\end{definition}

In literature, the spreading speed of delayed lattice differential equations has been investigated \cite{liang,thieme,weng}. In particular, if $f$  in \eqref{0} satisfies
\begin{description}
\item[(f1)] $f(u)=uh(u),h(M)=0$ for some $M >0;$
\item[(f2)] $h(u)$ is Lipschitz continuous and is decreasing for $u\in [0,M];$
\end{description}
then we can obtain the spreading speed of \eqref{0} by the theory in \cite{liang,thieme,weng}. More precisely, consider the following initial value problem
\begin{equation}\label{2.3}
\begin{cases}
\frac{du_n(t)}{dt}= [\mathcal{D}u]_n(x)
+f(u_n(t)),n\in \mathbb{Z}, t>0,  \\
u_n(0)=\psi(n), n\in \mathbb{Z}.
\end{cases}
\end{equation}
By Ma et al. \cite{mwz} and Weng et al. \cite{weng}, we have the following conclusions.
\begin{lemma}\label{le2.1}
Assume that (f1)-(f2) hold. If $0\le \psi(n)\le M, n\in \mathbb{Z},$ then \eqref{2.3} has a solution $u_n(t)$ for all $n\in \mathbb{Z},t>0.$
If ${w}_n(t), n\in \mathbb{Z},t>0,$ satisfies
\begin{equation*}
\begin{cases}
\frac{dw_n(t)}{dt}\ge (\le)[\mathcal{D}w]_n(t)
+f(w_n(t)),  \\
w_n(0)\ge (\le)\psi(n),
\end{cases}
\end{equation*}
or
\begin{eqnarray*}
w_{n}(t) &\geq &(\leq )e^{-(2D+d)(t-\theta)}w_n(\theta)  \nonumber \\
&&+%
\int_{\theta}^{t}e^{-(2D+d)(t-s)}[dw_{n}(s)+D(w_{n+1}(s)+w_{n-1}(s))+f(w_{n}(s))]ds
\label{2.5}
\end{eqnarray*}
with any fixed $d\ge 0,\theta\in [0,t),$ then ${w}_n(t) \ge (\le)u_n(t)$ for all $n\in \mathbb{Z},t>0.$ In particular, $w_n(x)$ is called {\rm an upper (a lower) solution} of \eqref{2.3}.
\end{lemma}
For $\lambda >0, c>0,$ define
\[
\Delta (\lambda, c)=D(e^{\lambda}+e^{-\lambda}-2)-c\lambda +h(0).
\]
\begin{lemma}\label{le2.2}
There exists $c_2=:\inf_{\lambda >0}\frac{D(e^{\lambda}+e^{-\lambda}-2)+ h(0)}{\lambda}>0$ such that
\begin{description}
\item[(1)]if $c>c_2,$ then $\Delta (\lambda, c)=0$ has two distinct positive real roots $\lambda_1(c)< \lambda_2(c)$ satisfying
\begin{equation*}
\Delta (\lambda, c)=
\begin{cases}
<0, \lambda \in (\lambda_1(c), \lambda_2(c)),\\
>0, \lambda \in (0, \lambda_1(c)) \text{ or } \lambda >\lambda_2(c);
\end{cases}
\end{equation*}
\item[(2)] $\Delta (\lambda, c)=0$ has no real roots for $c<c_2$;
\item[(3)] let $\varepsilon >0$ be any positive constant, then $c_2$ is continuous and strictly increasing in $h(0)\ge \varepsilon.$
\end{description}
\end{lemma}

\begin{lemma}\label{le2.3}
Assume that (f1)-(f2) hold. If $0\le \psi(n)\le M, n\in \mathbb{Z}$ such that $\psi_n(0)\ne 0$ for some $n\in\mathbb{Z}$ and  $\psi(n)=0$ for all large $|n|,$ then $c_2$ is the spreading speed of $u_n(t)$ defined by \eqref{2.3}.
\end{lemma}

\section{Asymptotic Spreading}
\noindent

In this section, we assume that (H1)-(H3) hold and consider the long time behavior of the following initial value problem
\begin{equation}\label{30}
\begin{cases}
\frac{du_n(t)}{dt}= [\mathcal{D}u]_n(x)
+u_n(t)g(u_n(t),u_n(t-\tau)),n\in \mathbb{Z}, t>0,  \\
u_n(s)=\varphi_n(s), n\in \mathbb{Z}, s\in [-\tau, 0],
\end{cases}
\end{equation}
in which $\varphi_n(s):[-\tau,0]\to l^{\infty}$ satisfies$0\le \varphi_n(s) \le 1$
and for each $n\in \mathbb{Z},$ it is uniform continuous in $s\in [-\tau, 0].$

In what follows, let $d>0$ be a constant such that
\[
du+ug(u,1), u\in [0,1]
\]
is monotone increasing. Consider
\[
\frac{du_n(t)}{dt}=-2D u_n(t)-du_n(t), u_n(0)=\omega (n), n\in\mathbb{Z},
\]
then we obtain an analytic and strictly positive semigroup in $l^{\infty}$ because of the boundedness of $2D+d$.
Then the standard semigroup theory implies the following results.
\begin{lemma}\label{le3.1}
Assume that (H1)-(H2) hold. \eqref{30} admits a unique mild solution $u_n(t)$ for all $t>0, n\in \mathbb{Z},$ which can be formulated by
\begin{equation}\label{2.6}
  u_n(t)=e^{-(2D+d)t}\varphi_n(0)+\int_0^t e^{-(2D+d)(t-s)}H_n(s)ds
\end{equation}
with
\[
H_n(s)=du_n(s)+D(u_{n+1}(s)+u_{n-1}(s))+u_n(s)g(u_n(s),u_n(s-\tau)).
\]
\end{lemma}

This lemma is also clear by the formula of constant variation, we omit the proof here. It should be noted that Lemma \ref{le3.1} remains true even if $\tau =0.$ By (H1)-(H2), $u_n(t)$ also satisfies the following conclusion.
\begin{lemma}\label{le3.2}
Assume that (H1)-(H2) hold. If $u_n(t)$ is defined by \eqref{30}, then
\[
0\le u_n(t)\le 1, t>0, n\in \mathbb{Z}.
\]
\end{lemma}

The positivity of $u_n(t)$ is clear by the quasipositivity of $ug(u,v),$ and $u_n(t)\le 1$ is clear by (H2) and Lemma \ref{le2.1}. Furthermore, using Lemmas \ref{le3.1}-\ref{le3.2}, we have the following conclusion.
\begin{lemma}\label{le3.3}
Assume that $u_n(t)$ is defined by \eqref{30} and (H1)-(H2) hold.
\begin{description}
  \item[(1)] For $t>\theta \ge 0, n\in \mathbb{Z},$ we have
  \[
u_n(t)\le e^{-(2D+d)(t-\theta)}w_n(\theta)+\int_\theta^t e^{-(2D+d)(t-s)}\overline{H}_n(s)ds
\]
with
\[
\overline{H}_n(s)=du_n(s)+ D(u_{n+1}(s)+u_{n-1}(s))+u_n(s)g(u_n(s),0).
\]
\item[(2)]
For $t>\theta \ge 0, n\in \mathbb{Z},$ we also have
  \[
u_n(t)\ge e^{-(2D+d)(t-\theta)}u_n(\theta)+\int_\theta^t e^{-(2D+d)(t-s)}\underline{H}_n(s)ds
\]
with
\[
\underline{H}_n(s)=du_n(s)+D(u_{n+1}(s)+u_{n-1}(s))+u_n(s)g(u_n(s),1).
\]
\end{description}
\end{lemma}

Since Lemma \ref{le3.1} also holds for $\tau =0,$ then $u_n(t)$ is an upper solution of
\begin{equation*}\label{31}
\begin{cases}
\frac{du_n(t)}{dt}= [\mathcal{D}u]_n(x)
+u_n(t)g(u_n(t),1),n\in \mathbb{Z}, t>0,  \\
u_n(0)=\varphi_n(0), n\in \mathbb{Z},
\end{cases}
\end{equation*}
and a lower solution of
\begin{equation}\label{32}
\begin{cases}
\frac{du_n(t)}{dt}= [\mathcal{D}u]_n(x)
+u_n(t)g(u_n(t),0),n\in \mathbb{Z}, t>0,  \\
u_n(0)=\varphi_n(0), n\in \mathbb{Z}.
\end{cases}
\end{equation}

By Lemmas \ref{le2.1} and \ref{le2.3}, we have the following conclusion.
\begin{lemma}\label{le3.4}
Assume that (H1)-(H3) hold. For any given $\epsilon >0.$ If $\varphi_0(0)\ge \epsilon$ and $0\le \varphi_n(0) \le 1,n\in\mathbb{Z},$ and $u_n(t)$ is defined by
\begin{equation*}\label{33}
\begin{cases}
\frac{du_n(t)}{dt}= [\mathcal{D}u]_n(x)
+u_n(t)g(u_n(t),1),n\in \mathbb{Z}, t>0,  \\
u_n(0)=\varphi_n(0), n\in \mathbb{Z}.
\end{cases}
\end{equation*}
Then there exists $\delta =\delta (\epsilon)>0$ such that
\[
u_0(t+\tau)> \delta, t>0,
\]
in which $\delta >0$ is independent of $\varphi_n(0), n\neq 0$.
\end{lemma}

Using Lemma \ref{le3.4}, we can obtain an auxiliary equation without time delay, which is formulated by the following lemma.
\begin{lemma}\label{le3.5}
Assume that $u_n(t)$ is defined by \eqref{30} and (H1)-(H3) hold. Then for each $\epsilon \in (0,1),$ there exists $M=M(\epsilon)\ge 1$ such that
\[
u_n(t)\ge e^{-(2D+d)(t-\theta)}u_n(\theta)+\int_\theta^t e^{-(2D+d)(t-s)}\underline{\underline{{H}}}_n(s)ds
\]
with $\theta \in [0,t)$ and
\[
\underline{\underline{{H}}}_n(s)=du_n(s)+D(u_{n+1}(s)+u_{n-1}(s))+u_n(s)g(Mu_n(s),\epsilon).
\]
\end{lemma}
\begin{proof}
If $u_n(t-\tau) < \epsilon,$ then
\[
u_n(t)g(u_n(t),u_n(t-\tau)) \ge u_n(t)g(u_n(t),\epsilon)
\]
from (H2). If $u_n(t-\tau) > \epsilon,$ then (H2)-(H3) and  Lemma \ref{le3.4} imply that there exists $M>1$ such that
\[
M u_n(t)\ge u_n(t-\tau)
\]
and
\[
u_n(t)g(u_n(t),u_n(t-\tau)) \ge u_n(t)g(M u_n(t),\epsilon).
\]
The proof is complete.
\end{proof}

We now present the main result of this section.
\begin{theorem}\label{th1}
Assume that  (H1)-(H3) hold and
\[
c^*= : \inf_{\lambda >0}\frac{D(e^{\lambda}+e^{-\lambda}-2)+ g(0,0)}{\lambda}.
\]
If $\varphi_n(s)$ satisfies
$
\varphi_n(s)=0, |n|>M, s\in [-\tau, 0]
$
with some $M >0$ and
$
\varphi_n(0)>0
$
for some $n\in \mathbb{Z},$ then $c^*$ is the spreading speed of $u_n(t)$ defined by \eqref{30}.
\end{theorem}
\begin{proof}
If $c>c^*,$ then $u_n(t)$ is a lower solution of \eqref{32} and
\[
\lim_{t\to \infty}\sup_{|n|> ct}u_n(t)=0
\]
by Lemma \ref{le2.3}. If $c'<c^*,$ then there exists $\epsilon >0$ such that
\[
D(e^{\lambda}+e^{-\lambda}-2)-c\lambda +g(0, \epsilon)>0
\]
for any $2c\le c'+c^*$ and $\lambda >0$ by Lemma \ref{le2.2}. Applying Lemmas \ref{le2.3} and \ref{le3.5}, we further obtain
\[
\liminf_{t\to \infty}\inf_{|n|< c't}u_n(t)>0.
\]
The proof is complete.
\end{proof}

If (H4) holds, we also have the following convergence conclusion.
\begin{theorem}\label{th2}
Assume that Theorem \ref{th1} holds and (H4) is true. Then
\[
\liminf_{t\to\infty}\inf_{|n|<ct}u_n(t)=\limsup_{t\to\infty}\sup_{|n|<ct}u_n(t)=E
\]
for any given $c< c^*.$
\end{theorem}
\begin{proof}
Define
\[
\liminf_{t\to\infty}\inf_{|n|<ct}u_n(t)=\underline{E}, \limsup_{t\to\infty}\sup_{|n|<ct}u_n(t)=\overline{E}.
\]
Then what we have done implies that
\[
0< \underline{E}\le \overline{E} \le 1.
\]
Using dominated convergence in \eqref{2.6}, we obtain
\[
\underline{E}\ge \frac{ D(\underline{E}+\underline{E})+d\underline{E}+\underline{E} g(\underline{E}, \overline{E})}{2D+d}
\]
and
\[
\overline{E}\le \frac{ D(\overline{E}+\overline{E})+d\overline{E}+\overline{E} g(\overline{E}, \underline{E})}{2D+d}.
\]
From (H4), the proof is complete.
\end{proof}

\begin{theorem}\label{th6}
Assume that (H1)-(H4) hold. If $u_n(t)$ is defined by \eqref{30} and $\phi_n(0)>0$ for some $n\in \mathbb{Z},$ then
\[
\liminf_{t\to\infty}\inf_{|n|<ct}u_n(t)=\limsup_{t\to\infty}\sup_{|n|<ct}u_n(t)=E
\]
for any given $c< c^*.$
\end{theorem}

The proof is similar to that of Theorem \ref{th2}, and we omit it here.
Before ending this section, we make the following remark.
\begin{remark}{\rm
The spreading speed of \eqref{3} with $\tau >0$ is the same as that of \eqref{3} with $\tau =0$, and we obtain the persistence of spreading speed of \eqref{3} with any time delay $\tau > 0$. For the corresponding topic in delayed reaction-diffusion equations, see Lin \cite{lin}.}
\end{remark}

\section{Minimal Wave Speed}
\noindent

In this part, we shall consider the traveling wave solutions of \eqref{3} and first present our main conclusion as follows.
\begin{theorem}\label{th3}
Assume that (H1)-(H3) hold. If $c\ge c^* (c<c^*),$ then  \eqref{3} has (has not) a positive traveling wave solution $\phi (\xi)$ such that
\begin{equation}\label{41}
\lim_{\xi\to -\infty}\phi (\xi)=0, 0< \liminf_{\xi\to \infty}\phi (\xi)\le \limsup_{\xi\to \infty}\phi (\xi)\le 1.
\end{equation}
Moreover, when (H4) and $c\ge c^*$ are true, then \eqref{2.2} remains true.
\end{theorem}

We now prove the result by three lemmas.
\begin{lemma}\label{le4.1}
Assume that (H1)-(H3) hold. If $c<c^*,$ then \eqref{2.1} has no positive solutions satisfying \eqref{41}.
\end{lemma}
\begin{proof}
Were the statement false, then there exists some $c''<c^*$ such that \eqref{2.1} with $c=c''$ has a positive solution satisfying \eqref{41}. Namely, $u(x,t)=\phi(x+c''t)$ also satisfies
\begin{equation*}\label{42}
\begin{cases}
\frac{du_n(t)}{dt}= [\mathcal{D}u]_n(x)
+u_n(t)g(u_n(t),u_n(t-\tau)),n\in \mathbb{Z}, t>0,  \\
u_n(s)=\phi (n+c''s), n\in \mathbb{Z}, s\in [-\tau, 0].
\end{cases}
\end{equation*}

Consider $-2n=(c+c'')t,$ then Theorem \ref{th2} implies that
\[
\liminf_{t\to\infty}\inf_{-2n=(c+c'')t}u_n(t) >0,
\]
which contradicts \eqref{41} because of
$
\xi= n+c''t\to -\infty, t\to\infty.
$
The proof is complete.
\end{proof}

\begin{lemma}\label{le4.2}
Assume that (H1)-(H3) hold. Then for each fixed $c> c^* ,$  \eqref{3} with \eqref{41} has a positive  solution $\phi (\xi)$.
\end{lemma}
\begin{proof}
If $\phi (\xi),\psi(\xi) \in C_{[0,1]},$ define an operator $F$ as follows
\[
F(\phi,\psi)(\xi)=\frac{1}{c}\int_{-\infty}^{\xi} e^{-\frac{(2D+d)}{c}(\xi-s)} H(\phi, \psi)(s)ds
\]
with
\[
H(\phi,\psi)(s)=d\phi(s)+D(\phi(s+c)+\phi(s-c))+\phi(s)g(\phi(s),\psi(s-c\tau)).
\]
When $\phi (\xi)=\psi(\xi),$ we also denote
\[
F(\phi,\phi)(\xi)=: P(\phi)(\xi).
\]
Then it is easy to prove that
$
P: C_{[0,1]} \to C_{[0,1]}.
$
In fact, because of
\[
0\le H(\phi,\psi)(s)\le d+2D,s\in \mathbb{R},
\]
then
\[
0\le P(\phi)(\xi)\le 1,\xi\in\mathbb{R},
\]
and the uniform continuity of $P(\phi)(\xi)$ is clear by the boundedness of $H(\phi,\psi)(s)$.

We now define two continuous functions as follows:
\[
\overline{\phi}(\xi)=\min\{e^{\lambda_1(c)\xi},1\}, \,\,
\underline{\phi}(\xi)=\max\{e^{\lambda_1(c)\xi}-qe^{\eta\lambda_1(c)\xi},0\}
\]
with $1< \eta <\min\{ 2, \lambda_2(c)/\lambda_1(c)\}$ and $q>1.$ Then
\begin{equation}\label{in}
\underline{\phi}(\xi)\le F(\underline{\phi}, \overline{\phi})(\xi)\le P(\phi)(\xi)\le F(\overline{\phi}, \underline{\phi})(\xi)\le \overline{\phi}(\xi),\xi \in\mathbb{R}
\end{equation}
if  $q>1$ is large enough and
\[
\phi(\xi)\in C_{[0,1]}, \underline{\phi}(\xi)\le \phi (\xi)\le \overline{\phi}(\xi).
\]

Let $4\mu< d/c$ be a constant and $\Gamma$ be
\[
\Gamma=\{\phi: \phi(\xi)\in C_{[0,1]}, \underline{\phi}(\xi)\le \phi (\xi)\le \overline{\phi}(\xi)\}.
\]
Then $\Gamma$ is convex and nonempty, and is bounded and closed in the sense of $|\cdot|_{\mu}.$ From \eqref{in}, we also obtain $P: \Gamma\to \Gamma.$
Moreover, the mapping is complete continuous in the sense of decay norm $|\cdot |_{\mu}.$
For the complete continuous of $P$, we can refer to Huang et al. \cite[Lemmas 3.3 and 3.5]{hlz} and Ma et al. \cite[Theorem 3.1]{mwz} since the proof is independent of the monotonicity.

Using Schauder's fixed point theorem, there is $\phi(\xi)\in \Gamma$ satisfying
\[
P(\phi)(\xi)=\phi(\xi), \underline{\phi}(\xi)\le \phi(\xi)\le \overline{\phi}(\xi),\xi\in\mathbb{R},
\]
which is also a solution of \eqref{3}.

Since $\phi(\xi)$ is a special positive  solution to \eqref{30}, then the asymptotic boundary condition is clear by what we have done in Section 3. The proof is complete.
\end{proof}
\begin{lemma}\label{le4.3}
Assume that (H1)-(H3) hold. If $c= c^* $ holds,  then  \eqref{3} has a positive  solution $\phi (\xi)$ satisfying \eqref{41}.
\end{lemma}
\begin{proof}
We now prove the result by passing to a limit function \cite{linruan}. Let $c_i\to c^*, i\in \mathbb{N},$ be strictly decreasing, then for each fixed $c_i,$ $P$ with $c=c_i$ has a positive fixed point $\phi_i(\xi)$ such that
\[
0<\phi_i(\xi)<1, \liminf_{\xi\to\infty}\phi_i(\xi)>0, \lim_{\xi\to -\infty}\phi_i(\xi)=0, i\in \mathbb{N}.
\]
Without loss of generality, we assume that
\[
\phi_i(0)= \delta, \phi_i(\xi) < \delta, \xi< 0
\]
with $g(4\delta, 1)>0.$ Due to the uniform boundedness of $\phi'_i(\xi),\xi\in \mathbb{R},$  $\phi_i(\xi)$ are equicontinuous. Using Ascoli-Arzela lemma, $\phi_i(\xi)$ has a subsequence, still denoted by $\phi_i(\xi),$ and there exists $\phi(\xi)\in C_{[0,1]} $ such that
\[
\phi_i(\xi) \to \phi(\xi), i\to\infty,
\]
in which the limit is pointwise and locally uniform on any bounded interval of $\xi\in\mathbb{R}.$ Clearly, we also have
\[
\phi(0)=\delta, \phi(\xi) \le \delta, \xi< 0.
\]

Note that
\[
e^{-\frac{(2D+d)}{c_i}(\xi -s)}\to e^{-\frac{(2D+d)}{c^*}(\xi -s)}, i\to\infty,
\]
and the convergence is uniform for $\xi\in\mathbb{R}, s\le \xi.$ Therefore, $\phi(\xi)$ is a fixed point of $P$
with $c=c^*.$ By the properties of $P$,  $\phi(\xi)$ is a positive solution to \eqref{3}.

Due to the conclusions in Section 3, the limit behavior of $\xi\to\infty$ is clear. We now consider the limit behavior when $\xi\to -\infty.$ If $\limsup_{\xi\to -\infty}\phi (\xi) >0,$ then there exists $\varepsilon_0>0$ such that there exists $- \xi_i <-i$ such that
\[
\phi (\xi_i) > \varepsilon_0, i \in \mathbb{N}.
\]

Since $\phi(\xi)$ is a special positive solution to \eqref{30},  then Theorem \ref{th6} implies that there exists $T$ independent of $i$ such that
\[
\phi (\xi_i+ T) > 3\delta,
\]
and a contradiction occurs when $i\to \infty.$ The proof is complete.
\end{proof}

To illustrate our main results, we consider the following example.
\begin{example}Assume that $r>0, a\in [0,1).$
Let
\[
c_*= \inf_{\lambda >0}\frac{D(e^{\lambda}+e^{-\lambda}-2)+ r}{\lambda}.
\]
Then $c_*$ is the minimal wave speed of traveling wave solutions connecting $0$ with $\frac{1}{1+a}$ of
\begin{equation}
\frac{du_n(t)}{dt}= [\mathcal{D}u]_n(x)
+r u_n(t) [1- u_n(t)-a u_n(t-\tau)],n\in \mathbb{Z}, t>0. \label{300}
\end{equation}
Moreover, $c_*$ is the spreading speed of the corresponding initial value problem of \eqref{300} if   $u_n(s)\ge 0, n\in\mathbb{N}, s\in [-\tau, 0]$ satisfies
\begin{description}
  \item[(I1)] for each $n\in\mathbb{N},$ $u_n(s)$ is continuous in $s\in [-\tau, 0];$
  \item[(I2)] $u_n(s)=0, |n|>M, s\in [-\tau, 0]$ with some $M >0;$
  \item[(I3)] $u_n(0)>0$ for some $n\in \mathbb{Z}.$
\end{description}
\end{example}

Before ending this paper, we make the following remark.
\begin{remark}{\rm
Although the delayed term reflect the intraspecific competition in population dynamics, the delay may be harmless to the propagation if the instantaneous competition dominates the delayed one (see (H3)).}
\end{remark}

\end{document}